\documentclass[11pt, reqno, twoside, letterpaper]{amsart}

\usepackage[
  letterpaper, twoside,
  inner=1.35in, outer=1.40in,
  top=1.25in,   bottom=1.25in,
  headsep=16pt, footskip=30pt,
]{geometry}

\usepackage{amsmath}
\usepackage{mathtools}

\usepackage{libertinus}
\usepackage{microtype}

\makeatletter
\g@addto@macro\normalsize{%
  \setlength\abovedisplayskip{13pt plus 3pt minus 4pt}%
  \setlength\belowdisplayskip{13pt plus 3pt minus 4pt}%
  \setlength\abovedisplayshortskip{0pt plus 3pt}%
  \setlength\belowdisplayshortskip{9pt plus 3.5pt minus 3pt}%
}
\makeatother

\usepackage{xcolor}
\usepackage{enumitem}
\usepackage{etoolbox}
\usepackage{csquotes}       

\usepackage{cite}

\numberwithin{equation}{section}

\theoremstyle{plain}
\newtheorem{X}{X}[section]
\newtheorem{theorem}[X]{Theorem}
\newtheorem{lemma}[X]{Lemma}

\newtheorem{proposition}[X]{Proposition}
\newtheorem*{theorem*}{Theorem}

\theoremstyle{definition}

\theoremstyle{remark}
\newtheorem{remark}[X]{Remark}
\newtheorem*{remark*}{Remark}

\allowdisplaybreaks[1]

\renewcommand{\le}{\leqslant}
\renewcommand{\ge}{\geqslant}

\newcommand{\proofstep}[1]{%
  \par\medskip
  \noindent\emph{#1.}
}

\makeatletter
\apptocmd{\thebibliography}{%
  \raggedright
  \@rightskip=\z@ \@plus 3em
  \rightskip=\@rightskip
  \parfillskip=\z@ \@plus 1fil
  \frenchspacing
}{}{\PackageWarning{preamble}{Could not patch thebibliography}}
\makeatother

\newenvironment{turn}[1]
  {\par\addvspace{0.7\baselineskip}%
   \noindent{\small\textsc{#1}}\par\nobreak\vspace{0.2\baselineskip}%
   \begingroup
   \small
   \leftskip=1.5em
   \rightskip=0pt
   \parindent=0pt
   \parskip=\medskipamount
   \relax}
  {\par\endgroup\addvspace{0.5\baselineskip}}

\usepackage{hyperref}
\hypersetup{
  unicode=true,
  pdftitle={Long runs of integers with small prime factors and the divisor function of n factorial},
  pdfauthor={Tristan Freiberg},
  pdfsubject={Number theory},
  pdfkeywords={prime gaps, smooth numbers, divisor function, factorials},
  pdfstartview={FitH},
  pdfmenubar=false,
  pdffitwindow=false,
  pdfnewwindow=true,
  bookmarksnumbered=true,
  linktoc=all,
  colorlinks=true,
  linkcolor={black},
  citecolor={black},
  filecolor={black},
  urlcolor={black},
}
\newcommand{\DOI}[1]{\href{https://doi.org/#1}{doi:#1}}
\newcommand{\ARXIV}[1]{\href{https://doi.org/10.48550/arXiv.#1}{arXiv:#1}}

\title[Long runs of integers with small prime factors]
      {Long runs of integers with small prime factors \\
       and the divisor function of $n!$}
\author[T. Freiberg]{Tristan Freiberg}
\address{Montr\'eal, Canada}
\date{\today}

\thanks{The core argument of Section~\ref{sec:proof-f-bound} was generated during an interaction with a generative-AI system. The author has independently verified the argument and takes full responsibility for the mathematical contents of this paper. See Section~\ref{sec:ai-provenance} for details and Appendix~\ref{app:record} for documentation of the interaction.}

\begin{document}

\begin{abstract}
Let $d$ be the divisor function, and let $K(n)$ be the least positive integer $K$ for which $d((n + K)!) \ge 2d(n!)$. Erd\H{o}s, Graham, Ivi\'c and Pomerance proved that, for infinitely many $n$,
\begin{equation*}
K(n) > (1/9)(\log n)(\log_{2} n)(\log_{4} n)/(\log_{3} n)^3.
\end{equation*} 
We improve upon this by a factor of order $\log_{3} n$, which brings the bound to the same order as Rankin's 1938 lower bound for gaps between consecutive primes. The two problems are closely related, but a long prime-free interval does not by itself produce a large value of $K(n)$: what is needed is a weighted variant of the Erd\H{o}s--Rankin construction. We follow the method of Erd\H{o}s, Graham, Ivi\'c and Pomerance, replacing a key estimate by an averaging argument that permits some integers to remain uncovered. This improvement was formulated and proved during a private interaction with Claude Fable 5.1, a publicly available generative-AI system; the argument is verified and presented here by the author.
\end{abstract}

\maketitle

\section{Introduction}
 \label{sec:intro}

The number of positive divisors $d(n)$ of a positive integer $n$ can fluctuate sharply between consecutive integers. An application of Dirichlet's theorem on primes in arithmetic progressions gives
\begin{equation}
\label{eq:d-ratio-limits}
\liminf_{n \to \infty} \frac{d(n + 1)}{d(n)} = 0 \qquad \text{and} \qquad \limsup_{n \to \infty} \frac{d(n + 1)}{d(n)} = \infty.
\end{equation}

By contrast, the sequence $d(n!)$ is strictly increasing, and its successive ratios exhibit considerably more regular behavior. Erd\H{o}s, Graham, Ivi\'c and Pomerance \cite[\S 3]{EGIP1996} proved, among several results related to $d(n!)$, that
\begin{equation*}
\frac{d((n + 1)!)}{d(n!)} \to 1
\end{equation*}
as $n \to \infty$ through a set of asymptotic density $1$. On the other hand, if $n + 1$ is prime,
\begin{equation*}
\frac{d((n + 1)!)}{d(n!)} = 2.
\end{equation*}

This suggests asking how many steps are required for $d(n!)$ to double. Accordingly, define
\begin{equation}
\label{eq:k-def}
K(n) := \min\left\{K \ge 1 : \frac{d((n + K)!)}{d(n!)} \ge 2 \right\}.
\end{equation}
If $p$ is the least prime exceeding $n$, then the step from $(p - 1)!$ to $p!$ contributes a factor of $2$, so 
\begin{equation*}
K(n) \le p - n,
\end{equation*}
which shows the minimum in \eqref{eq:k-def} exists. Thus, large values of $K(n)$ can only occur when $n$ is followed by a long prime-free interval. The converse is not automatic, since the smaller increases at composite arguments may accumulate to a factor of $2$ before the next prime is reached. 

Erd\H{o}s, Graham, Ivi\'c and Pomerance \cite[Corollary 2]{EGIP1996} proved that for infinitely many integers $n$,
\begin{equation}
\label{eq:egip-k-bound}
K(n) > (\log n) \frac{(\log_{2} n)(\log_{4} n)}{9(\log_{3} n)^3}.
\end{equation}
(Here and throughout, $\log_{j}$ denotes the $j$\,th iterated logarithm; see Section~\ref{sec:notation}.) Our main result improves this lower bound by a factor of order $\log_{3} n$.
\begin{theorem}
\label{thm:main}
Let $K(n)$ be as in \eqref{eq:k-def}. Given any $\epsilon > 0$, there are infinitely many $n$ such that 
\begin{equation*}
K(n) > \left(\frac{3}{4\pi^2} - \epsilon\right)(\log n) \frac{(\log_{2} n) (\log_{4} n)}{(\log_{3} n)^2}.
\end{equation*}
\end{theorem}

The core new argument in this paper was generated during a private interaction with Claude Fable~5.1, a publicly available generative-AI system. More precisely, the argument proving Proposition~\ref{prop:f-bound}, from which Theorem~\ref{thm:main} is deduced, originated in that interaction. The author has independently checked the argument, verified the mathematical results on which it relies, and takes full responsibility for the contents of this paper. The provenance and verification of the argument are discussed in Section~\ref{sec:ai-provenance}; documentation of the interaction is given in Appendix~\ref{app:record}, with the complete records of both sessions supplied as ancillary files.

The key ingredient in the proof of \eqref{eq:egip-k-bound} is a result concerning the sum
\begin{equation}
\label{eq:s-def}
S(n) := \sum_{p^{a} \, \| \, n} ap
\end{equation}
of prime factors of $n$ counted with multiplicity. Specifically, Erd\H{o}s, Graham, Ivi\'c and Pomerance \cite[Theorem 3]{EGIP1996} proved that if
\begin{equation}
\label{eq:f-def}
f(n) := \min\left\{h \ge 1 : \sum_{t = 1}^{h} S(n + t) > n\right\},
\end{equation}
then for any $\epsilon > 0$,
\begin{equation}
\label{eq:egip-f-bound}
f(n) \ge \left(\frac{1}{4} - \epsilon\right)(\log n)\frac{(\log_{2} n)(\log_{4} n)}{(\log_{3} n)^3}
\end{equation}
for infinitely many $n$. Their Corollary~2 follows from this together with an elementary upper bound for $d(n!)/d((n - 1)!)$ in terms of $S(n)$, which we record as Lemma~\ref{lem:ratio-bound} below. 

We prove Theorem~\ref{thm:main} using the same deduction, but improve the lower bound in \eqref{eq:egip-f-bound} by a factor of order $\log_{3} n$.

\begin{proposition}
\label{prop:f-bound}
Let $S(n)$ and $f(n)$ be as in \eqref{eq:s-def} and \eqref{eq:f-def}. Given any $\epsilon > 0$, there are infinitely many $n$ such that 
\begin{equation*}
f(n) \ge \left(\frac{3}{2\pi^2} - \epsilon\right)(\log n)\frac{(\log_{2} n)(\log_{4} n)}{(\log_{3} n)^2}.
\end{equation*}
\end{proposition}

After fixing notation and recalling relevant background, we deduce Theorem~\ref{thm:main} from Proposition~\ref{prop:f-bound} in Section~\ref{sec:deduction}, and prove the proposition in Section~\ref{sec:proof-f-bound}. Section~\ref{sec:ai-provenance} and Appendix~\ref{app:record} document the provenance of the argument.

\section{Notation}
\label{sec:notation}

Throughout, $p$ and $q$ denote primes. For $a \ge 1$, the notation $p^a \, \| \, n$ means that $p^a$ exactly divides $n$, that is, $p^a \mid n$ but $p^{a + 1} \nmid n$. For integers $a$ and $b$, $(a,b)$ denotes their greatest common divisor. We write $\pi(x)$ for the number of primes not exceeding $x$, $\phi$ for Euler's totient function, and $\zeta$ for the Riemann zeta function. The notation $\#T$ denotes the cardinality of a finite set $T$.

The divisor function $d(n)$ is defined in Section~\ref{sec:intro}, as are $K(n)$, $S(n)$ and $f(n)$; see \eqref{eq:k-def}, \eqref{eq:s-def} and \eqref{eq:f-def}. Given a sufficiently large real number $x$, the $j$\,th iterated logarithm $\log_j x$ is defined by
\begin{equation*}
\log_1 x := \log x, \qquad \log_{j} x := \log(\log_{j - 1} x) \quad (j \ge 2),
\end{equation*}
where $\log$ denotes the natural logarithm.

For functions $g$ and $G$ of a parameter tending to infinity, the notation $g = o(G)$ means that 
\begin{equation*}
\frac{g}{G} \to 0
\end{equation*}
as that parameter tends to infinity. In particular, $g = o(1)$ means that $g \to 0$. We write $g \sim G$ if $g/G \to 1$, equivalently if $g = (1 + o(1))G$.

The notation $F = O(G)$, equivalently $F \ll G$ or $G \gg F$, means that
\begin{equation*}
|F|\le C|G|
\end{equation*}
for some positive constant $C$ and all sufficiently large values of the relevant parameter. Unless stated otherwise, the implied constants are absolute. Finally, $F \asymp G$ means that $F \ll G$ and $G \ll F$. Other notation is introduced \emph{in situ}.

\section{Related results and prime gaps}
\label{sec:background}

The local behavior of $d(n)$ appearing in \eqref{eq:d-ratio-limits} has been studied much more precisely. Erd\H{o}s \cite{ERD1986} conjectured that the sequence $d(n + 1)/d(n)$ is dense in $(0,\infty)$. Partial results concerning its set of limit points were obtained by Hildebrand \cite{HIL1987}. Using sieve results of Goldston, Graham, Pintz and Y{\i}ld{\i}r{\i}m \cite{GGPY2011}, Hasanalizade \cite{HAS2021} sharpened Hildebrand's result, and further progress was made by Schlage-Puchta \cite{SCH2025}. Building on Hasanalizade's approach and again exploiting the Goldston--Graham--Pintz--Y{\i}ld{\i}r{\i}m sieve, Eberhard \cite{EBE2026} proved the stronger statement that $d(n + 1)/d(n)$ assumes every positive rational value infinitely often.

The problem considered here is also closely related to that of large gaps between consecutive primes. Let $p_m$ denote the $m$\,th smallest prime, so that 
\begin{equation*}
G(x) := \max_{p_{m + 1} \le x}(p_{m + 1} - p_m)
\end{equation*}
is the largest gap between consecutive primes up to $x$. Erd\H{o}s \cite{ERD1935} proved the lower bound
\begin{equation*}
G(x) \gg \frac{(\log x)(\log_2 x)}{(\log_3 x)^2},
\end{equation*}
and Rankin \cite{RAN1938} refined his construction to obtain
\begin{equation}
\label{eq:rankin-gap}
G(x) \gg \frac{(\log x)(\log_2 x)(\log_4 x)}{(\log_3 x)^2}.
\end{equation}
The method underlying these results, now generally known as the Erd\H{o}s--Rankin construction, produces a long interval of composite integers by choosing residue classes modulo suitably selected primes.

For many decades, subsequent improvements to \eqref{eq:rankin-gap} affected only the implied constant. This barrier was broken independently by Maynard \cite{MAY2016} and by Ford, Green, Konyagin and Tao \cite{FGKT2016}. Maynard showed that the implicit constant in Rankin's bound may be taken arbitrarily large, while Ford, Green, Konyagin and Tao showed that Rankin's lower bound may be multiplied by a factor tending to infinity. Their subsequent joint work with Maynard \cite{FGKMT2018} gave the quantitative improvement
\begin{equation*}
G(x) \gg \frac{(\log x)(\log_2 x)(\log_4 x)} {\log_3 x}.
\end{equation*}

Very recently, two further improvements to \eqref{eq:rankin-gap} have been announced. Neither has been refereed, and we record both as claims. A manuscript attributed to GPT~5.6~Sol \cite{GPT2026} asserts
\begin{equation*}
G(x) \gg \frac{(\log x)(\log_2 x)}{\log_4 x},
\end{equation*}
and a preprint released by OpenAI \cite{OAI2026}, which attributes its proof to GPT~6~Astra, asserts the stronger bound
\begin{equation*}
G(x) \gg \frac{(\log x)(\log_2 x)^2(\log_4 x)}{(\log_3 x)^2}.
\end{equation*}
A Lean~4 formalization of the main theorem of each is available, and both developments are discussed on the Erd\H{o}s Problems forum, where T.~F.~Bloom also gives an expository sketch of the ideas underlying the first \cite{EP4DISC}. We use neither bound in the present paper, for reasons discussed at the end of this section.

The connection between large prime gaps and $K(n)$ is immediate but only one-sided. As observed in Section~\ref{sec:intro}, if $p > n$ is the least prime exceeding $n$, then
\begin{equation*}
K(n) \le p - n.
\end{equation*}
Thus a large value of $K(n)$ requires a long prime-free interval. A long prime gap alone, however, does not suffice: even in the absence of primes, the successive contributions from composite integers may accumulate so that $d((n + k)!)/d(n!)$ reaches $2$ before the next prime is encountered. 

For this reason, the lower bound of Erd\H{o}s, Graham, Ivi\'c and Pomerance in \eqref{eq:egip-k-bound} is not obtained merely by applying an Erd\H{o}s--Rankin large-gap result. Rather, their proof adapts the underlying Erd\H{o}s--Rankin construction to a stronger, weighted problem. In the notation of \eqref{eq:s-def} and \eqref{eq:f-def}, it is not enough to arrange that the integers in a long interval are composite; one must also control the cumulative quantity
\begin{equation*}
\sum_{t = 1}^{H} S(n + t).
\end{equation*}
Their Theorem~3, recalled in \eqref{eq:egip-f-bound}, achieves precisely such control and, together with \eqref{eq:egip-ratio-bound}, yields their lower bound for $K(n)$. Thus the problem of obtaining large values of $K(n)$ may be viewed as a weighted analog of the large-prime-gap problem.

The improvements to \eqref{eq:rankin-gap} of Maynard \cite{MAY2016}, Ford--Green--Konyagin--Tao \cite{FGKT2016}, and Ford--Green--Konyagin--Maynard--Tao \cite{FGKMT2018} improve the treatment of the final covering stage of the Erd\H{o}s--Rankin construction, in which each position surviving the earlier sieving must be covered, and \cite{OAI2026} likewise replaces that stage by a stronger covering result. In the weighted problem considered here that stage is not needed at all: the argument of Section~\ref{sec:proof-f-bound} leaves the surviving positions uncovered and controls their contribution on average over a further parameter. What limits the attainable interval length is instead the number of survivors together with the average cost of one, and the greedy sieve already supplies the required order of magnitude at the scale used in our proof. Improvements to the final covering therefore do not by themselves transfer. The tilted construction of \cite{GPT2026} modifies an earlier sieving stage rather than the final one, and is of a different character; we do not investigate here whether it can be adapted to the weighted problem.

Our Proposition~\ref{prop:f-bound} improves this weighted Erd\H{o}s--Rankin estimate. In particular, whereas the bound of Erd\H{o}s, Graham, Ivi\'c and Pomerance has size
\begin{equation*}
\frac{(\log n)(\log_2 n)(\log_4 n)}{(\log_3 n)^3},
\end{equation*}
our result restores the classical Rankin scale
\begin{equation*}
\frac{(\log n)(\log_2 n)(\log_4 n)}{(\log_3 n)^2},
\end{equation*}
up to the explicit constant appearing in Proposition~\ref{prop:f-bound}.

\section{Deduction of the main theorem}
\label{sec:deduction}

Erd\H{o}s, Graham, Ivi\'c and Pomerance \cite[Corollary~2]{EGIP1996} deduce their result in five lines. We give their argument with some details made explicit, and formulate it as a general transfer principle: a lower bound for $f(n)$ on any scale that is slowly varying at its own order yields the corresponding lower bound for $K(n)$, with the constant halved. Any future improvement of Proposition~\ref{prop:f-bound} on such a scale therefore yields a corresponding improvement of Theorem~\ref{thm:main} without further argument.

The deduction rests on an elementary inequality relating $d(n!)/d((n - 1)!)$ to $S(n)$, established in \cite[(5)]{EGIP1996}. We reproduce it with proof for completeness.

\begin{lemma}
\label{lem:ratio-bound}
For all $n \ge 1$,
\begin{equation}
\label{eq:egip-ratio-bound}
\frac{d(n!)}{d((n - 1)!)} \le \exp\left(\frac{S(n)}{n}\right).
\end{equation}
\end{lemma}

\begin{proof}
For a prime $p$ write $\nu_p(m)$ for the exponent of $p$ in $m!$, so that
\begin{equation*}
d(m!) = \prod_{p} (\nu_p(m) + 1).
\end{equation*}
Since $n! = n(n - 1)!$, we have
\begin{equation*}
\nu_p(n) = \nu_p(n - 1) + a
\end{equation*}
if $p^a \,\|\, n$, while $\nu_p(n) = \nu_p(n - 1)$ if $p \nmid n$. Hence all factors corresponding to primes not dividing $n$ cancel in the quotient, and
\begin{equation*}
\frac{d(n!)}{d((n - 1)!)}
= \prod_{p^a \, \| \, n} \left(1 + \frac{a}{\nu_p(n - 1) + 1}\right)
\le \exp\left(\sum_{p^a \, \| \, n} \frac{a}{\nu_p(n - 1) + 1} \right).
\end{equation*}
Since
\begin{equation*}
\nu_p(n - 1) \ge \left\lfloor \frac{n - 1}{p} \right\rfloor \ge \frac{n - 1 - (p - 1)}{p},
\end{equation*}
we have $p(\nu_p(n - 1) + 1) \ge n$, whence
\begin{equation*}
\frac{a}{\nu_p(n - 1) + 1} \le \frac{ap}{n}.
\end{equation*}
Summing over $p^a \, \| \, n$ gives the lemma.
\end{proof}

We now formulate the deduction in a form that does not depend on the particular lower bound supplied by Proposition~\ref{prop:f-bound}.

\newpage
\begin{proposition}
\label{prop:transfer}
Let $\Lambda$ be a positive real-valued function defined for all sufficiently large real $x$, and suppose that
\begin{equation}
\label{eq:transfer-lambda-growth}
\Lambda(x) \to \infty \qquad (x \to \infty),
\end{equation}
and that, for each fixed $C > 0$,
\begin{equation}
\label{eq:transfer-lambda-stability}
\Lambda(x + h) \sim \Lambda(x) \qquad (x \to \infty)
\end{equation}
uniformly for $0 \le h \le C\Lambda(x)$. Let $c_1 > 0$ and suppose that
\begin{equation}
\label{eq:transfer-f-hypothesis}
f(n) \ge c_1 \Lambda(n)
\end{equation}
for infinitely many $n$. Then, for every $\epsilon > 0$,
\begin{equation}
\label{eq:transfer-k-conclusion}
K(n) > \left( \frac{c_1}{2} - \epsilon \right) \Lambda(n)
\end{equation}
for infinitely many $n$.
\end{proposition}

\begin{proof}
Since $K(n) \ge 1$, the conclusion is trivial unless $\epsilon < c_{1}/2$, and we may therefore assume this. Set
\begin{equation*}
c = \frac{c_{1}}{2} - \frac{\epsilon}{2} > 0.
\end{equation*}

Let $N$ be one of the infinitely many integers for which \eqref{eq:transfer-f-hypothesis} holds, and put
\begin{equation*}
H = \left\lfloor c \Lambda(N) \right\rfloor.
\end{equation*}
Since $\Lambda(N) \to \infty$, we have $H \ge 1$ for all sufficiently large such $N$. Moreover
\begin{equation*}
2H \le 2c\Lambda(N) = (c_1 - \epsilon)\Lambda(N) < c_1 \Lambda(N) \le f(N),
\end{equation*}
so that $2H < f(N)$. By the definition \eqref{eq:f-def} of $f(N)$, every $h < f(N)$ satisfies $\sum_{t = 1}^{h} S(N + t) \le N$; taking $h = 2H$,
\begin{equation*}
\sum_{t = 1}^{2H} S(N + t) \le N.
\end{equation*}
Splitting this sum into two blocks of length $H$, at least one of
\begin{equation*}
\sum_{t = 1}^{H} S(N + t), \qquad \sum_{t = H + 1}^{2H} S(N + t)
\end{equation*}
is at most $N/2$.

In the first case put $n = N$, and in the second case put $n = N + H$; in either case put $K = H$. Since $n \ge N$, we have in either case
\begin{equation}
\label{eq:transfer-small-s-block}
\sum_{t = 1}^{K} S(n + t) \le \frac{N}{2} \le \frac{n}{2}.
\end{equation}

We next compare $\Lambda(n)$ with $\Lambda(N)$. In either case $n = N + h$ with $0 \le h \le H \le c\Lambda(N)$, so \eqref{eq:transfer-lambda-stability}, applied with $C = c$, gives
\begin{equation*}
\Lambda(n) \sim \Lambda(N) \qquad (N \to \infty).
\end{equation*}
Since $\Lambda(N) \to \infty$, we also have $H \sim c\Lambda(N)$, and therefore
\begin{equation*}
\frac{K}{\Lambda(n)} = \frac{H}{\Lambda(n)} \to c = \frac{c_1}{2} - \frac{\epsilon}{2}
\end{equation*}
as $N \to \infty$ along these integers. In particular, for all sufficiently large such $N$,
\begin{equation}
\label{eq:transfer-k-lower-bound}
K > \left( \frac{c_1}{2} - \epsilon \right) \Lambda(n).
\end{equation}

Finally, by Lemma~\ref{lem:ratio-bound} and
\eqref{eq:transfer-small-s-block},
\begin{align*}
\frac{d((n + K)!)}{d(n!)}
& =
\prod_{t = 1}^{K} \frac{d((n + t)!)}{d((n + t - 1)!)}
\\
& \le
\exp\left( \sum_{t = 1}^{K} \frac{S(n + t)}{n + t} \right)
\\
& \le
\exp\left( \frac{1}{n} \sum_{t = 1}^{K} S(n + t) \right)
\\
& \le
\exp\left( \frac{1}{2} \right)
\\
& 
< 2.
\end{align*}
Hence $K(n) > K$, and \eqref{eq:transfer-k-lower-bound} gives \eqref{eq:transfer-k-conclusion}. Since the admissible $N$ are unbounded and the corresponding integers $n$ satisfy $n \ge N$, this holds for infinitely many $n$.
\end{proof}

\begin{remark}
\label{rem:halving}
The factor $1/2$ in Proposition~\ref{prop:transfer} is intrinsic to this use of $f(n)$ as a black box. Indeed, if
\begin{equation*}
m = f(N) - 1,
\end{equation*}
then the definition of $f$ gives only
\begin{equation*}
\sum_{t = 1}^{m} S(N + t) \le N.
\end{equation*}
Splitting these $m$ terms into two consecutive blocks shows that one block of length essentially $m/2$ has sum at most $N/2$. On the other hand, no longer block can be guaranteed from the displayed inequality alone. If $K > m/2$, then all intervals of $K$ consecutive indices in $\{1,\ldots,m\}$ have a common index. An arbitrary nonnegative sequence with total mass $N$ concentrated at that index satisfies the displayed inequality, while every such block has sum $N$, so no bound of the form $\theta N$ with $\theta < 1$ can be guaranteed. Thus any improvement of the factor $1/2$ would require additional information about the values $S(N + t)$, or a modification of the definition \eqref{eq:f-def} of $f$.
\end{remark}

\begin{proof}[Proof of Theorem~\ref{thm:main}]
Write
\begin{equation*}
\mathcal{L}(x) = (\log x) \frac{(\log_{2} x)(\log_{4} x)}{(\log_{3} x)^{2}},
\end{equation*}
defined for all sufficiently large $x$. Plainly $\mathcal{L}(x) \to \infty$ as $x \to \infty$, which verifies \eqref{eq:transfer-lambda-growth}. Moreover, $\mathcal{L}(x) = o(x)$. To verify \eqref{eq:transfer-lambda-stability}, fix $C > 0$ and let $0 \le h \le C\mathcal{L}(x)$. Since $\mathcal{L}(x) = o(x)$, we have $h = o(x)$, so $x + h = (1 + o(1))x$ and hence
\begin{equation*}
\log(x + h) = \log x + o(1) \sim \log x,
\end{equation*}
uniformly in $h$. Iterating, $\log_{j}(x + h) \sim \log_{j} x$ for each fixed $j \ge 1$, uniformly in $h$, and therefore
\begin{equation*}
\mathcal{L}(x + h) \sim \mathcal{L}(x) \qquad (x \to \infty),
\end{equation*}
uniformly for $0 \le h \le C\mathcal{L}(x)$.

Since $K(n) \ge 1$, the assertion of Theorem~\ref{thm:main} is trivial unless $\epsilon < 3/(4\pi^{2})$, and we may therefore assume this. By Proposition~\ref{prop:f-bound}, there are infinitely many $n$ with
\begin{equation*}
f(n) \ge \left( \frac{3}{2\pi^{2}} - \epsilon \right)\mathcal{L}(n).
\end{equation*}
Apply Proposition~\ref{prop:transfer} with $\Lambda = \mathcal{L}$,
\begin{equation*}
c_{1} = \frac{3}{2\pi^{2}} - \epsilon > 0,
\end{equation*}
and $\epsilon/2$ in place of $\epsilon$. There are then infinitely many $n$ with
\begin{equation*}
K(n)
>
\left( \frac{c_{1}}{2} - \frac{\epsilon}{2} \right)\mathcal{L}(n)
=
\left( \frac{3}{4\pi^{2}} - \epsilon \right)\mathcal{L}(n),
\end{equation*}
which is the assertion of Theorem~\ref{thm:main}.
\end{proof}

\section{Proof of Proposition~\ref{prop:f-bound}}
\label{sec:proof-f-bound}

Since the result is trivial for $\epsilon \ge 3/(2\pi^2)$, it suffices to consider
\begin{equation*}
0 < \epsilon < \frac{3}{2\pi^2}.
\end{equation*}
Let $u$ be a real parameter tending to infinity. Throughout this proof, all asymptotic notation refers to $u \to \infty$, with $\epsilon$ (and hence $\delta$) fixed. Put 
\begin{equation*}
M := \prod_{(\log u)^2 \, \le p \, \le u} p.
\end{equation*}
By the prime number theorem and Mertens' theorem for products,
\begin{equation}
\label{eq:m-estimates}
\log M = (1 + o(1)) u \qquad \text{and} \qquad \frac{M}{\phi(M)} = \left(\frac{1}{2} + o(1)\right)\frac{\log u}{\log_{2} u}.
\end{equation}
Also define
\begin{align}
y & := u^{ (1 - \epsilon)\log_{3} u / \log_{2} u }, \label{eq:y-def} \\
c_0 & := \frac{3}{\pi^2}(1 - \epsilon)^2 = \frac{(1 - \epsilon)^2}{2\zeta(2)}, \label{eq:c0-def} \\
L & := \left\lfloor \frac{c_0 u \log u \log_{3} u}{(\log_{2} u)^2} \right\rfloor, \label{eq:l-def} \\
\delta & := \frac{\epsilon}{2}, \label{eq:delta-def} \\
J & := [(1 - \delta)M, M] \cap \mathbb{Z}. \label{eq:j-def}
\end{align}

We use the usual terminology of the Erd\H{o}s--Rankin construction. For a prime $p$, we say that the residue class $a_p \bmod p$ \emph{covers} an integer $t$ if
\begin{equation*}
t \equiv a_p \bmod p.
\end{equation*}
An integer is \emph{covered} by a collection of residue classes if it is covered by at least one of them, and is \emph{uncovered} otherwise. After the residue classes have been chosen, the Chinese remainder theorem will be used to choose $A$ such that
\begin{equation*}
A \equiv -a_p \bmod p.
\end{equation*}
Thus every covered $t$ satisfies $p \mid A + t$ for at least one of the corresponding primes $p$.

\proofstep{Step 1 (initial sieving)} For $y < p \le u$, put $a_p = 0$, and define the residual set
\begin{equation*}
T_1 := [1,L] \setminus \bigcup_{y \, < \, p \, \le \, u} \{t : t \equiv 0 \bmod p\}.
\end{equation*}
Thus $T_1$ consists precisely of those $t \in [1,L]$ having no prime factor in $(y,u]$. Since $L < u^2$ for all sufficiently large $u$, no integer $t \le L$ can have two prime factors exceeding $u$. Hence every $t \in T_1$ is either $y$-smooth or is of the form
\begin{equation*}
t = mp,
\end{equation*}
where $p > u$ is prime. In the latter case, since $mp \le L$,
\begin{equation*}
m \le \frac{L}{p} < \frac{L}{u}.
\end{equation*}
Therefore the number of integers of this form is at most
\begin{align}
\label{eq:large-prime-factor-bound}
\begin{split}
\sum_{m \, \le \, L/u} \pi\left(\frac{L}{m}\right) 
& = 
\sum_{m \, \le \, L/u} (1 + o(1))\frac{L/m}{\log(L/m)} \\
& \le
(1 + o(1)) \frac{L}{\log u} \sum_{m \, \le \, L/u}\frac{1}{m} \\
& \le
(1 + o(1)) \frac{L}{\log u} \left(1 + \log(L/u)\right) \\
& =
(1 + o(1))\frac{L\log_{2} u}{\log u}.
\end{split}
\end{align}
Here the prime number theorem is used uniformly for $m \le L/u$, since $L/m \ge u$, and we have also used
\begin{equation*}
\sum_{m \, \le \, x} \frac{1}{m} \le 1 + \log x
\end{equation*}
together with
\begin{equation*}
\log(L/u) = (1 + o(1))\log_2 u,
\end{equation*}
which follows from \eqref{eq:l-def}.

To estimate the number of $y$-smooth integers in $T_1$, write
\begin{equation*}
\Psi(L, y) := \#\{t \le L : p \mid t \implies p \le y\},
\end{equation*}
and set
\begin{equation*}
v := \frac{\log L}{\log y} = (1 + o(1)) \frac{\log_{2} u}{(1 - \epsilon)\log_{3} u},
\end{equation*}
where the last estimate follows from \eqref{eq:y-def} and \eqref{eq:l-def}.

We first record the size of $\rho(v)$, where $\rho$ denotes the Dickman--de Bruijn function. Since $v \to \infty$, by \cite[(1.8)]{DEB1951},
\begin{equation*}
\rho(v) = \exp\left(-v\log v - v\log_{2} v + O(v)\right).
\end{equation*}
Only the first term in the exponent contributes at the scale relevant here. Now
\begin{equation*}
\log v = \log_{3} u - \log_{4} u + O(1)
\end{equation*}
and
\begin{equation*}
\log_{2} v = \log_{4} u + o(1),
\end{equation*}
so
\begin{equation*}
v\log v = \left(\frac{1}{1 - \epsilon} + o(1)\right)\log_{2} u,
\end{equation*}
while
\begin{equation*}
v\log_{2} v + O(v) = o(\log_{2} u).
\end{equation*}
Consequently,
\begin{equation}
\label{eq:rho-power-log}
\rho(v) = (\log u)^{-1/(1 - \epsilon) + o(1)}.
\end{equation}

We now invoke de Bruijn's estimates \cite[(1.3), (1.4)]{DEB1951}, which together give
\begin{equation*}
\Psi(L, y) = L\rho(v)\left\{1 + O\left(\frac{\log(2 + v)}{\log y}\right)\right\} + O\left(Lv^{2}R(y)\right),
\end{equation*}
uniformly for $1 \le v < \log y$, the additive constant appearing in \cite[(1.4)]{DEB1951} being omissible in this range by the footnote there. Here $R$ may be taken to be $R(y) = C\exp\{-c(\log y)^{1/2}\}$ for suitable positive constants $c$ and $C$; see \cite[(3.7), (3.8)]{DEB1951}. By \eqref{eq:l-def} we have $\log L = (1 + o(1))\log u$, so \eqref{eq:y-def} gives
\begin{equation*}
\log y = (\log u)^{1 + o(1)},
\end{equation*}
whereas $v = (\log_{2} u)^{1 + o(1)}$. In particular $v < \log y$ for all sufficiently large $u$, and the relative error above is $O(\log_{3} u / (\log u)^{1 + o(1)}) = o(1)$. Moreover
\begin{equation*}
v^{2}R(y) = \exp\left(-(\log u)^{1/2 + o(1)}\right),
\end{equation*}
which is smaller than any fixed power of $1/\log u$, and hence is $o(\rho(v))$ by \eqref{eq:rho-power-log}. Therefore
\begin{equation*}
\Psi(L, y) = (1 + o(1))L\rho(v),
\end{equation*}
and only the resulting upper bound will be used. Since $\epsilon > 0$ is fixed and
\begin{equation*}
\frac{1}{1 - \epsilon} > 1 + \frac{\epsilon}{2},
\end{equation*}
it follows from \eqref{eq:rho-power-log} that
\begin{equation}
\label{eq:psi-bound}
\Psi(L,y) \ll \frac{L}{(\log u)^{1 + \epsilon/2}}.
\end{equation}

Combining \eqref{eq:large-prime-factor-bound} with \eqref{eq:psi-bound}, we obtain
\begin{equation}
\label{eq:t1-bound}
\#T_1 \le (1 + o(1))\frac{L\log_{2} u}{\log u}.
\end{equation}

\proofstep{Step 2 (greedy covering)} For any finite set of integers $T$ and any prime $p$, 
\begin{equation*}
 \# T = \sum_{a \bmod p} \# \{t \in T : t \equiv a \bmod p\},
\end{equation*}
so there exists an integer $a_p$ such that 
\begin{equation*}
 \#\{t \in T : t \equiv a_p \bmod p\} \ge \frac{\# T}{p}.
\end{equation*} 
Proceeding through the primes $p \in [(\log u)^2,y]$ one at a time, choose $a_p$ so that its residue class contains at least a proportion $1/p$ of the elements not already covered. We obtain a residual set 
\begin{equation*}
 T_2 = T_1 \setminus \bigcup_{(\log u)^2 \le p \le y} \{t \in T_1 : t \equiv a_p \bmod p\}
\end{equation*}
whose cardinality satisfies the bound
\begin{align}
\label{eq:t2-bound}
\begin{split}
\#T_2 
& \le 
\#T_1 \prod_{(\log u)^2 \, \le \, p \, \le \, y} \left( 1 - \frac{1}{p} \right) \\
& = 
\#T_1 (1 + o(1))\frac{2\log_{2} u}{\log y} \\
& \le 
(1 + o(1))\frac{2L(\log_{2} u)^2}{(\log y)(\log u)} \\
& = 
\left(\frac{2c_0}{1 - \epsilon} + o(1)\right) \frac{u\log_{2} u}{\log u}.
\end{split}
\end{align}
Here we have used Mertens' theorem for products, \eqref{eq:t1-bound}, and the definitions in  \eqref{eq:y-def} and \eqref{eq:l-def}.

Let $A$ be the unique integer in $[0, M - 1]$ given by the Chinese remainder theorem such that
\begin{equation*}
A \equiv -a_p \bmod p \qquad \text{for every } p \mid M.
\end{equation*}
Then, by construction,
\begin{equation}
\label{eq:covering-properties}
\begin{aligned}
(A + t, M) & = 1 && \text{for } t \in T_2, \\
(A + t, M) & > 1 && \text{for } t \in [1,L] \setminus T_2.
\end{aligned}
\end{equation}

\proofstep{Step 3 (averaging over $j$)} In the classical Erd\H{o}s--Rankin construction, one now uses the remaining primes to cover the residual set one element at a time. Here we leave the elements of $T_2$ uncovered and instead control their contribution on average over a further parameter $j$.

Write
\begin{equation*}
n_{j,\, t} := jM + A + t.
\end{equation*}
Our aim is to show that for some $j \in J$, 
\begin{equation*}
\sum_{t = 1}^{L} S(n_{j,\, t}) \le n_{j, 0}.
\end{equation*}
By the definition of $f$, this will imply $f(n_{j, 0}) > L$. We shall obtain such a $j$ by bounding the above sum on average over $j \in J$.

For $j \in J$ and $0 \le t \le L$,
\begin{equation}
\label{eq:njt-size}
(1 - \delta)M^2 \le n_{j,\, t} \le (1 + o(1))M^2.
\end{equation}
The lower bound follows immediately from the definition of $J$ in \eqref{eq:j-def}. For the upper bound, we have
\begin{equation*}
n_{j,\, t} \le M^2 + M + L.
\end{equation*}
Moreover, \eqref{eq:l-def} gives $L = u^{1 + o(1)}$, while \eqref{eq:m-estimates} gives $\log M = (1 + o(1))u$. Thus $L = o(M)$, and hence
\begin{equation*}
M + L = o(M^2),
\end{equation*}
which proves the upper bound in \eqref{eq:njt-size}.

Define 
\begin{equation*}
S_1(n) := \sum_{\substack{p^a \, \| \, n \\ p \, > \, n/u^3}} ap 
\qquad \text{and} \qquad 
S_2(n) := \sum_{\substack{p^a \, \| \, n \\ p \, \le \, n/u^3}} ap,
\end{equation*}
so that we have the decomposition 
\begin{equation}
\label{eq:decomposition-of-s}
\sum_{j \in J} \sum_{t = 1}^{L} S(n_{j,\, t}) = \sum_{j \in J} \sum_{t = 1}^{L} S_1(n_{j,\, t}) + \sum_{j \in J} \sum_{t = 1}^{L} S_2(n_{j,\, t}).
\end{equation}

We first bound the second sum on the right of \eqref{eq:decomposition-of-s}. For any $n$,
\begin{equation*}
S_2(n) \le \frac{n}{u^3} \sum_{p^a \, \| \, n} a \le \frac{n\log n}{u^3\log 2}.
\end{equation*}
By \eqref{eq:njt-size} and the first estimate in \eqref{eq:m-estimates}, uniformly for $j \in J$ and $1 \le t \le L$,
\begin{equation*}
S_2(n_{j,\, t}) \ll \frac{M^2}{u^2}.
\end{equation*}
Moreover,
\begin{equation*}
\#J = \delta M + O(1) = (1 + o(1))\delta M.
\end{equation*}
Consequently,
\begin{equation}
\label{eq:s2-sum-trivial-bound}
\sum_{j \in J} \sum_{t = 1}^{L} S_2(n_{j,\, t}) \ll \frac{\delta LM^3}{u^2} = o(\delta M^3),
\end{equation}
where the last estimate follows from \eqref{eq:l-def}.

To bound the sum of $S_1$ in the decomposition \eqref{eq:decomposition-of-s}, we first consider the contribution from the covered elements $t \in [1,L] \setminus T_2$. Suppose that $S_1(n_{j,\, t}) > 0$. By \eqref{eq:njt-size} and the first estimate in \eqref{eq:m-estimates}, we have $n_{j,\, t} > u^6$ for all sufficiently large $u$. Hence every prime contributing to $S_1(n_{j,\, t})$ exceeds 
\begin{equation*} 
\frac{n_{j,\, t}}{u^3} > \sqrt{n_{j,\, t}}. 
\end{equation*} 
It follows that there is a unique such prime $q$, and that it occurs to the first power. Thus 
\begin{equation*} 
n_{j,\, t} = q\ell, \qquad \ell < u^3, 
\end{equation*} 
and, by \eqref{eq:njt-size}, 
\begin{equation} 
\label{eq:s1-large-prime} 
S_1(n_{j,\, t}) = q = \frac{n_{j,\, t}}{\ell} \le (1 + o(1)) \frac{M^2}{\ell}. 
\end{equation}

Since $t$ is covered, \eqref{eq:covering-properties} gives a prime $p \mid M$ such that $p \mid A + t$. As $p \mid jM$, we also have $p \mid n_{j,\, t}$. Moreover, $p \le u$, whereas
\begin{equation*}
q > \frac{n_{j,\, t}}{u^3} > u
\end{equation*}
for all sufficiently large $u$, by \eqref{eq:njt-size} and the first estimate in \eqref{eq:m-estimates}. Hence $p \ne q$, and from $p \mid q\ell$ it follows that $p \mid \ell$. Since every prime divisor of $M$ is at least $(\log u)^2$, we therefore have
\begin{equation*}
\ell \ge (\log u)^2.
\end{equation*}

For a positive integer $\ell \le u^3$, let 
\begin{equation*} 
\mathcal{N}(\ell) := \#\left\{ (j,t) \in \mathbb{Z}^2 : \frac{M}{2} \le j \le M,\ 1 \le t \le L,\ \frac{jM + A + t}{\ell} \text{ is prime} \right\}. 
\end{equation*} 
We shall use the following elementary sieve estimate of Erd\H{o}s, Graham, Ivi\'c and Pomerance \cite[Lemma~2]{EGIP1996}, stated here in our notation. Their proof uses $L$ only through the count of admissible $t$, and is therefore uniform in $L$ over the range needed here. 

\begin{lemma} 
\label{lem:egip-count} 
Uniformly for positive integers $\ell \le u^3$, 
\begin{equation*} 
\mathcal{N}(\ell) \ll \frac{LM\log u}{\phi(\ell)u\log_{2} u} + \frac{(\ell, M) M\log u}{\phi(\ell)u\log_{2} u}. 
\end{equation*} 
\end{lemma}

We shall also use the elementary estimates
\begin{equation}
\label{eq:phi-sum-estimates}
\sum_{\ell \ge x}\frac{1}{\ell\phi(\ell)} \ll \frac{1}{x},
\qquad
\sum_{\ell \ge 1}\frac{(\ell,M)}{\ell\phi(\ell)} \ll \frac{M}{\phi(M)}.
\end{equation}
For the first, one may use partial summation with
\begin{equation*}
\sum_{\ell \le x}\frac{\ell}{\phi(\ell)} \ll x.
\end{equation*}
For the second, writing $d = (\ell, M)$ and $\ell = dk$, and using $\phi(dk) \ge \phi(d)\phi(k)$, gives
\begin{equation*}
\sum_{\ell \ge 1}\frac{(\ell,M)}{\ell\phi(\ell)}
\le
\sum_{d\mid M}\frac{1}{\phi(d)} \sum_{k\ge1}\frac{1}{k\phi(k)} 
\ll
\sum_{d\mid M}\frac{1}{\phi(d)}
=
\prod_{p \mid M}\left(1 + \frac{1}{p - 1}\right)
=
\frac{M}{\phi(M)}.
\end{equation*}

Since $\delta < 1/2$, we have $J \subseteq [M/2,M]$. Hence, combining \eqref{eq:s1-large-prime}, the lower bound $\ell \ge (\log u)^2$, Lemma~\ref{lem:egip-count}, \eqref{eq:phi-sum-estimates}, and the second estimate in \eqref{eq:m-estimates}, we obtain
\begin{align}
\label{eq:s1-covered-bound}
\begin{split}
\sum_{j \in J}
\sum_{t \in [1,L] \setminus T_2} S_1(n_{j,\, t})
& \le
(1 + o(1))M^2 \sum_{(\log u)^2 \le \ell < u^3} \frac{\mathcal{N}(\ell)}{\ell} \\
& \ll
\frac{LM^3\log u}{u\log_2u} \sum_{\ell \ge (\log u)^2} \frac{1}{\ell\phi(\ell)} \\
& \qquad
+
\frac{M^3\log u}{u\log_2u} \sum_{\ell \ge (\log u)^2} \frac{(\ell,M)}{\ell\phi(\ell)} \\
& \ll
\frac{LM^3}{u\log u\log_2u} + \frac{M^3(\log u)^2}{u(\log_2u)^2} \\
& =
o(M^3).
\end{split}
\end{align}

Indeed, by \eqref{eq:l-def},
\begin{equation*}
L \ll \frac{u\log u\log_3u}{(\log_2u)^2},
\end{equation*}
so the first term above is
\begin{equation*}
\ll M^3 \frac{\log_3u}{(\log_2u)^3} = o(M^3),
\end{equation*}
while the second is plainly $o(M^3)$.

We now consider the contribution from the uncovered elements $t \in T_2$ to the sum of $S_1$ in \eqref{eq:decomposition-of-s}. This is the point at which our argument departs from \cite{EGIP1996}. The key new ingredient is the following estimate.

\begin{lemma}
\label{lem:uncovered-s1-bound}
Uniformly for $t \in T_2$,
\begin{equation}
\label{eq:uncovered-s1-bound}
\sum_{j \in J} S_1(n_{j,\, t}) \le \left(\zeta(2) + o(1)\right) \frac{\delta M^3\log u}{u\log_2 u}.
\end{equation}
\end{lemma}

\begin{proof}
Fix $t \in T_2$. If $S_1(n_{j,\, t}) > 0$, then, as above, there is a unique prime $q$ contributing to $S_1(n_{j,\, t})$, it occurs to the first power, and
\begin{equation*}
n_{j,\, t} = q\ell,\qquad \ell < u^3.
\end{equation*}
Moreover,
\begin{equation*}
S_1(n_{j,\, t}) = q = \frac{n_{j,\, t}}{\ell} \le (1 + o(1)) \frac{M^2}{\ell},
\end{equation*}
uniformly for $j \in J$ and $t \in T_2$.

Since $t$ is uncovered, \eqref{eq:covering-properties} gives
\begin{equation*}
(n_{j,\, t},M) = (A + t,M) = 1.
\end{equation*}
It follows from $n_{j,\, t} = q\ell$ that
\begin{equation*}
(\ell,M) = 1.
\end{equation*}

Fix a positive integer $\ell < u^3$ with $(\ell, M) = 1$. Since $M$ is invertible modulo $\ell$, the divisibility condition
\begin{equation*}
\ell \mid jM + A + t
\end{equation*}
restricts $j$ to a single residue class modulo $\ell$. If no $j \in J$ lies in this class there is nothing to bound for this $\ell$, so we may assume otherwise and take $j_{\ell}$ to be the least such $j$. Since $J$ is a set of consecutive integers, the values of $j \in J$ in this residue class are then precisely
\begin{equation*}
j = j_{\ell} + \ell m, \qquad 0 \le m \le m_{\ell},
\end{equation*}
for some nonnegative integer $m_{\ell}$. As $j_{\ell} \ge (1 - \delta)M$ and $j_{\ell} + \ell m_{\ell} \le M$, we have $\ell m_{\ell} \le \delta M$, so $m$ runs through at most
\begin{equation*}
\frac{\delta M}{\ell} + 1
\end{equation*}
consecutive integers. For such $j$, the quotient $n_{j,\, t}/\ell$ is an integer, and
\begin{equation*}
\frac{n_{j,\, t}}{\ell} = \frac{j_{\ell} M + A + t}{\ell} + mM.
\end{equation*}
Thus the integers $n_{j,\, t}/\ell$ arising from $j \in J$ in this residue class lie in an arithmetic progression with modulus $M$, whose initial term $(j_{\ell} M + A + t)/\ell$ is a positive integer. The values of $q$ to be counted are prime terms of this progression.

Put
\begin{equation*}
b_{\ell} := \frac{j_{\ell} M + A + t}{\ell}.
\end{equation*}
Since
\begin{equation*}
\ell b_{\ell} \equiv A + t \bmod M
\end{equation*}
and both $\ell$ and $A + t$ are coprime to $M$, we have
\begin{equation*}
(b_{\ell}, M) = 1.
\end{equation*}
The terms $b_{\ell} + mM$ with $0 \le m \le m_{\ell}$ are contained in an interval of length
\begin{equation*}
Y_{\ell} := \left(\frac{\delta M}{\ell} + 1\right)M.
\end{equation*}
Uniformly for $\ell < u^3$,
\begin{equation*}
Y_{\ell} = (1 + o(1))\frac{\delta M^2}{\ell}
\end{equation*}
and
\begin{equation*}
\log\left(\frac{Y_{\ell}}{M}\right) = \log\left(\frac{\delta M}{\ell} + 1\right) = (1 + o(1))u,
\end{equation*}
by the first estimate in \eqref{eq:m-estimates}. In particular, $Y_{\ell} > M$ for all sufficiently large $u$.

The Brun--Titchmarsh theorem of Montgomery and Vaughan \cite[Theorem~2]{MV1973} therefore shows that the number of primes $q$ arising for this fixed $\ell$ is at most
\begin{equation*}
\frac{2Y_{\ell}} {\phi(M)\log(Y_{\ell}/M)} \le (2 + o(1)) \frac{\delta M^2}{\ell\phi(M)u}.
\end{equation*}
Multiplying by the bound
\begin{equation*}
S_1(n_{j,\, t}) \le (1 + o(1))\frac{M^2}{\ell}
\end{equation*}
and summing over $\ell$, we obtain
\begin{align*}
\sum_{j \in J} S_1(n_{j,\, t})
& \le
(2 + o(1)) \frac{\delta M^4}{\phi(M)u} \sum_{\substack{\ell < u^3 \\ (\ell, M) = 1}}\frac{1}{\ell^2} \\
& \le
(2\zeta(2) + o(1)) \frac{\delta M}{\phi(M)}\cdot \frac{M^3}{u} \\
& =
\left(\zeta(2) + o(1)\right) \frac{\delta M^3\log u}{u\log_{2} u},
\end{align*}
where in the last line we used the second estimate in \eqref{eq:m-estimates}. This proves \eqref{eq:uncovered-s1-bound}.
\end{proof}

Summing Lemma~\ref{lem:uncovered-s1-bound} over $t \in T_2$ and using \eqref{eq:t2-bound}, we obtain
\begin{align}
\label{eq:s1-uncovered-bound}
\begin{split}
\sum_{j \in J}\sum_{t \in T_2} S_1(n_{j,\, t})
& \le
\#T_2 \left(\zeta(2) + o(1)\right) \frac{\delta M^3\log u}{u\log_{2} u} \\
& \le
\left( \frac{2\zeta(2)c_0}{1 - \epsilon} + o(1) \right) \delta M^3 \\
& =
\left(1 - \epsilon + o(1)\right)\delta M^3,
\end{split}
\end{align}
where in the last line we used
\begin{equation*}
2\zeta(2)\frac{3}{\pi^2} = 1
\end{equation*}
and the definition of $c_0$ in \eqref{eq:c0-def}.

Combining \eqref{eq:s2-sum-trivial-bound}, \eqref{eq:s1-covered-bound}, and \eqref{eq:s1-uncovered-bound}, and recalling that $\delta$ is fixed, we conclude that
\begin{equation}
\label{eq:total-s-bound}
\sum_{j \in J} \sum_{t = 1}^{L} S(n_{j,\, t}) \le \left(1 - \epsilon + o(1)\right) \delta M^3.
\end{equation}
In contrast, the corresponding argument of Erd\H{o}s, Graham, Ivi\'c and Pomerance gives an $o(M^3)$ bound for the total sum. Our larger choice of $L$ leads instead to a contribution of full order $M^3$; the saving is instead in the constant, since $(1 - \epsilon)\delta < (1 - \delta)\delta$.

Call $j \in J$ \emph{bad} if
\begin{equation*}
\sum_{t = 1}^{L} S(n_{j,\, t}) > n_{j, 0} = jM + A.
\end{equation*}
Since $j \in J$ and $A \ge 0$, every bad $j$ contributes more than
\begin{equation*}
(1 - \delta)M^2
\end{equation*}
to the left-hand side of \eqref{eq:total-s-bound}. Hence the number of bad $j$ is at most
\begin{equation*}
\left(\frac{1 - \epsilon}{1 - \delta} + o(1)\right) \delta M.
\end{equation*}
Since $\delta = \epsilon/2$ by \eqref{eq:delta-def}, we have
\begin{equation*}
\frac{1 - \epsilon}{1 - \delta} < 1.
\end{equation*}
On the other hand,
\begin{equation*}
\#J = \delta M + O(1).
\end{equation*}
It follows that, for all sufficiently large $u$, not every $j \in J$ is bad. Thus there is some $j \in J$ for which
\begin{equation*}
\sum_{t = 1}^{L} S(n_{j,\, t}) \le n_{j, 0}.
\end{equation*}
Putting
\begin{equation*}
n := n_{j, 0} = jM + A,
\end{equation*}
we obtain from the definition of $f$ that
\begin{equation}
\label{eq:f-greater-than-l}
f(n) > L.
\end{equation}

It remains to express $L$ in terms of $n$. Since
\begin{equation*}
(1 - \delta) M^2\le n < M^2 + M,
\end{equation*}
we have
\begin{equation*}
\log n = (2 + o(1)) \log M = (2 + o(1))u
\end{equation*}
by the first estimate in \eqref{eq:m-estimates}. Hence
\begin{equation*}
u = \left(\frac{1}{2} + o(1)\right) \log n.
\end{equation*}
It follows from \eqref{eq:l-def} that
\begin{equation}
\label{eq:l-in-terms-of-n}
L = \left(\frac{c_0}{2} + o(1)\right) \frac{(\log n)(\log_{2} n)(\log_{4} n)}{(\log_{3} n)^2}.
\end{equation}
Moreover,
\begin{equation*}
\frac{c_0}{2} = \frac{3(1 - \epsilon)^2}{2\pi^2} > \frac{3}{2\pi^2} - \epsilon,
\end{equation*}
since
\begin{equation*}
\frac{3(1 - \epsilon)^2}{2\pi^2} - \left(\frac{3}{2\pi^2} - \epsilon\right) = \epsilon \left(1 - \frac{3}{\pi^2}\right) + \frac{3\epsilon^2}{2\pi^2} > 0.
\end{equation*}
Thus, for all sufficiently large $u$, \eqref{eq:f-greater-than-l} and \eqref{eq:l-in-terms-of-n} give
\begin{equation*}
f(n) \ge \left(\frac{3}{2\pi^2} - \epsilon\right) \frac{(\log n)(\log_{2} n)(\log_{4} n)}{(\log_{3} n)^2}.
\end{equation*}
Finally, $n \to \infty$ as $u \to \infty$, so this holds for infinitely many $n$. The proposition follows.

\section{AI provenance and verification}
\label{sec:ai-provenance}

\subsection*{Provenance of the argument}
On 8~September~2026 the author submitted to Claude Fable~5.1 (Anthropic; model string \texttt{claude-fable-5-1}), through the Claude web interface, the problem of improving the lower bound for $f(n)$ in the theorem of Erd\H{o}s, Graham, Ivi\'c and Pomerance discussed above. The prompt included an excerpt of \cite{EGIP1996} containing their theorem and proof, as well as an excerpt from the generative-AI preprint \cite{OAI2026}. In response, Claude Fable~5.1 formulated the improvement stated here as Proposition~\ref{prop:f-bound} and supplied a proof. The prompt and response are supplied as the ancillary files \texttt{prompt.txt} and \texttt{response.txt}, and Appendix~\ref{app:record} reproduces the parts of the session relevant to the argument.

Most of that proof follows the argument of Erd\H{o}s, Graham, Ivi\'c and Pomerance: in particular, the Erd\H{o}s--Rankin sieving construction, the decomposition $S = S_1 + S_2$, and the treatment of the covered positions are adaptations of their proof. The principal new ingredient in the generated argument is the decision to leave the final set $T_2$ of surviving positions uncovered and instead average their contribution over the parameter $j$. The corresponding estimate is Lemma~\ref{lem:uncovered-s1-bound}. Together with the surrounding averaging argument, this permits the larger choice of $L$ and yields Proposition~\ref{prop:f-bound}.

The response presents this step as an adaptation of the preprint \cite{OAI2026}, describing it as a transfer of the idea of leaving the last surviving positions uncovered and allowing a free translation parameter to absorb them. The two arguments share that idea but diverge in what the free parameter must deliver: in \cite{OAI2026} every surviving position must be made composite, which requires a second-moment argument and auxiliary primes not dividing the modulus, whereas here it suffices that $\sum_{t \le L} S(n_{j,\, t})$ be small on average, for which a first moment over $j$ suffices. The proof given here therefore uses no result of \cite{OAI2026}. The response also declined to rely on \cite[Proposition~1.2]{OAI2026}, and hence on \cite[Theorem~1.1]{OAI2026} itself; see Appendix~\ref{app:record}. The model's account of the idea's origin should be read together with the second session described below.

The author selected the problem and the source material supplied to the model. After receiving the generated argument, the author checked the proof, rewrote it in his own exposition, changed notation where appropriate, supplied additional details, and verified the external results used in the final argument. The deduction of Theorem~\ref{thm:main} from Proposition~\ref{prop:f-bound} was reconstructed from the corresponding argument of Erd\H{o}s, Graham, Ivi\'c and Pomerance and formulated here as the more general transfer principle of Proposition~\ref{prop:transfer}.

Since the model presents its argument as an adaptation of an idea drawn from \cite{OAI2026}, it is natural to ask whether a comparable improvement would have been produced from the excerpt of \cite{EGIP1996} alone. To test this, on 11~September~2026 the author submitted a second prompt, in a fresh session with no access to the first, containing the excerpt of \cite{EGIP1996} and no other source, and asking for an improvement to the lower bound for $f(n)$ without reference to \cite{OAI2026}. The model produced a purported proof of a lower bound of the same order as Proposition~\ref{prop:f-bound}, with the smaller constant $1/300$ in place of $3/(2\pi^2) - \epsilon$, by a different argument.

Rather than leaving positions uncovered, the second argument covers every position and replaces Lemma~\ref{lem:egip-count} by a weighted estimate in which the covering enters through
\begin{equation*}
\sum_{t \le L} \frac{\phi(d_t)}{d_t^2}, \qquad d_t := (A + t, M),
\end{equation*}
with a probabilistic construction used to make this weight small. The uncovered positions of Section~\ref{sec:proof-f-bound} correspond to the case $d_t = 1$. Thus the second session provides evidence that the model could obtain an improvement of this order without being supplied with \cite{OAI2026}. The author has not verified the second argument in detail; no claim concerning its correctness is made here, and it is not used in this paper. The second prompt and response are supplied as the ancillary files \texttt{prompt-02.txt} and \texttt{response-02.txt}.

\subsection*{Verification}
The generated proof was not incorporated into the paper without further checking. The author checked the argument line by line, rewrote it in his own notation and exposition, and supplied details where the generated response was compressed or incomplete. In particular, the parameter choices and their ranges, the estimates for the successive sieving stages, the decomposition $S= S_1 + S_2$, the treatment of the covered positions, the estimate for the uncovered positions in Lemma~\ref{lem:uncovered-s1-bound}, the averaging over $j$, and the final conversion from the auxiliary parameter $u$ to $n$ were checked in detail. Constants and dependencies among the asymptotic parameters were also checked throughout.

The external ingredients used in the final proof were checked against their original sources, including the relevant estimates of de~Bruijn \cite{DEB1951}, the Erd\H{o}s--Graham--Ivi\'c--Pomerance argument \cite{EGIP1996} and the auxiliary lemma on which the covered contribution is based, and the form of the Brun--Titchmarsh theorem due to Montgomery and Vaughan \cite{MV1973} used in Lemma~\ref{lem:uncovered-s1-bound}. The deduction of Theorem~\ref{thm:main} from Proposition~\ref{prop:f-bound} was reconstructed from the corresponding argument of Erd\H{o}s, Graham, Ivi\'c and Pomerance. False starts and claims appearing in the generative session that were not used in the final proof were discarded. The author therefore regards the argument printed in this paper, rather than the model output, as the mathematical object for which he takes responsibility.

\subsection*{Authorship and responsibility}
The language model is not listed as an author. The author takes full responsibility for the mathematical claims, exposition, citations, and other contents of this paper.

\subsection*{Reproducibility}
Language-model outputs are not deterministic, and models may be updated or withdrawn. The record in Appendix~\ref{app:record} documents the interaction that actually occurred; it is not a protocol that can be expected to reproduce the same output. The complete prompts and responses from both sessions are supplied as ancillary files. In the response files, mathematical notation has been converted from LaTeX markup to plain-text Unicode notation for readability, without changes to the substantive content.

\subsection*{Purpose}
This paper is in part an experiment in the use of language models for mathematical research. The author's view is that substantial use of such systems should be disclosed rather than left implicit, and that mathematical claims obtained with their assistance should be subjected to the same standard of verification as claims obtained by other means: the author must check the argument and take responsibility for what is published.

\newpage
\appendix

\section{Record of the generative session}
\label{app:record}

This appendix documents the generative-AI interaction described in Section~\ref{sec:ai-provenance}. It is not part of the mathematical argument of the paper. In particular, statements made by the language model below are reproduced as part of the historical record and are not thereby asserted to be correct. The mathematical argument for which the author takes responsibility is the argument presented in the body of the paper.

For readability, the lengthy excerpts from published or otherwise cited sources that were included in the prompt are not reproduced here. They are replaced by brief descriptions enclosed in square brackets. Apart from these replacements and the omission of portions not relevant to the development of the argument, the quoted prompt and response are reproduced as they occurred, including errors. The complete prompt and response are supplied as the ancillary plain-text files \texttt{prompt.txt} and \texttt{response.txt}. 

The response uses $\theta$ for the parameter called $\delta$ in this paper, $U$ for the uncovered set $T_2$, $i$ for the index $t$, and $n_{i,j}$ for $n_{j,\, t}$. 

The following session with Anthropic's Claude Fable~5.1 took place on 8~September~2026. In the prompt, Theorem~1.1 and Proposition~1.2 refer to results in \cite{OAI2026}.

\begin{turn}{Prompt}
Excerpts from one paper and one preprint are shown below. One might hope to improve the lower bound for f(n) given in Theorem 3 by using Theorem 1.1 or the ideas underneath it. Do this, with a theorem statement and proof, or explain why you cannot.

[Excerpt~1: the statement and proof of Theorem~3 of Erd\H{o}s, Graham, Ivi\'c and Pomerance \cite{EGIP1996}, including their Erd\H{o}s--Rankin construction, Lemma~2, the decomposition $S = S_1 + S_2$, and the proof of their lower bound for $f(n)$.]

[Excerpt~2: the statement of Theorem~1.1 and Proposition~1.2 of \cite{OAI2026}, together with the deduction of Theorem~1.1 from Proposition~1.2 in \cite{OAI2026}. The proof of \cite[Proposition~1.2]{OAI2026} was not included.]

\end{turn}

\begin{turn}{Claude Fable 5.1}
\textelp{}

The usable idea: do not cover the last survivors; leave them and control them on average over $j$.

\textelp{}

\textbf{Lemma A.} For $i \in U$,
\begin{equation*}
\sum_{j \in J} S_1(n_{i,j}) \le \left(\zeta(2) + o(1)\right) \frac{\theta M^3\log u}{u\log_2u}.
\end{equation*}

\textelp{}

[Further portions of the response develop this estimate, combine it with the Erd\H{o}s--Rankin construction and the treatment of the covered positions, and deduce the improved lower bound for $f(n)$. The complete response is supplied in \texttt{response.txt}.]
\end{turn}

\begin{turn}{Author's note}
The response declined to use Proposition~1.2 of \cite{OAI2026}, reproduced in the second excerpt above, on the grounds that only its statement had been supplied and that, in the range in which it is applied, it appeared to the model to be beyond the reach of current methods. The response also reports an unsuccessful attempt to locate a public record of \cite{OAI2026}; in fact a Lean~4 formalization, linked from that preprint, and an independent expository sketch \cite{EP4DISC} were both available at the time. The author regards the model's refusal to rely on an unsupported proposition as appropriate, and notes that the failure of retrieval, rather than of reasoning, was what left it unsupported.

The proof ultimately produced in the response does not use \cite[Proposition~1.2]{OAI2026}. Its essential new step is instead the averaging argument for the uncovered positions, represented above by Lemma~A and appearing in the paper as Lemma~\ref{lem:uncovered-s1-bound}.

A second session, on 11~September~2026, supplied the excerpt of \cite{EGIP1996} alone. It produced a purported proof of a bound of the same order, described in Section~\ref{sec:ai-provenance}. The prompt and response are supplied as \texttt{prompt-02.txt} and \texttt{response-02.txt}; they are not reproduced here.
\end{turn}

\end{document}